\documentclass{amsart}
\usepackage{amsmath,amsthm,amssymb,amsthm, amsfonts}
\makeatletter
 \def\@makefnmark{%
 \leavevmode
 \raise.9ex\hbox{\check@mathfonts
 \fontsize\sf@size\z@\normalfont%
 \@thefnmark}%
 }
 \makeatother
\newcommand\diam{\operatorname{diam}}
\newcommand{\N}{\mathbb{N}}

\newcommand{\R}{\mathbb{R}}

\newcommand{\Z}{\mathbb{Z}}

\newcommand{\Ray}{\R_{+}}
\theoremstyle{definition}
\newtheorem{theorem}{Theorem}[section]
\newtheorem{lemma}[theorem]{Lemma}
\newtheorem{proposition}[theorem]{Proposition}
\newtheorem{corollary}[theorem]{Corollary}
\theoremstyle{definition}

\theoremstyle{remark}
\newtheorem{remark}[theorem]{Remark}
\newcounter{cn}
\title{Indecomposable continua as Higson coronae}
\author{Yutaka Iwamoto}
\address{Department of Engineering Science, National Institute of Technology (KOSEN), Niihama College,
Niihama, 792-8580, Japan}
\email{iwamoto@sci.niihama-nct.ac.jp}
\begin{document}

\begin{abstract}
In this paper, we consider spaces whose Higson coronae are indecomposable continua.
We show that for a non-compact proper metric space $X$
 which is coarsely geodesic and has coarse bounded geometry,
 the Higson corona of $X$ is an indecomposable continuum
 if and only if $X$ is coarsely equivalent to the space of natural numbers.
Then we give characterizations of finitely generated groups that have one or two ends
 by decomposability/indecomposability of the components of their Higson coronae.
\end{abstract}

\maketitle

\section{Introduction}\label{intro}
The notion of Higson compactification was introduced by Higson for non-compact complete Riemannian manifolds.
Then Roe defined the Higson compactification for more general spaces (cf. \cite{Roe0}, \cite{Roe}).
It has been used in large scale geometry
 to capture the global properties of Riemannian manifolds, groups, and others.
In particular, the global information is thought to be condensed in its Higson corona.
\par
In this paper, we consider spaces whose Higson coronae are indecomposable continua.
For example, the ray $\Ray=[0,\infty)$ is known as such a space, that is,
 the Higson corona $\nu \Ray$ of $\Ray$ is known to be
  a non-metrizable indecomposable continuum \cite{IT}.
This is closely related to the fact that
 the Stone-\v{C}ech remainder $\beta\Ray\setminus \Ray$
 of the ray is a non-metrizable indecomposable continuum \cite{Bellamy},
 where $\beta\Ray$ denotes the Stone-\v{C}ech compactification of $\Ray$.
For a non-compact locally connected generalized continuum $X$, Dickman \cite{DICK} gave a characterization that the Stone-\v{C}ech remainder $\beta X\setminus X$ is an indecomposable continuum if and only if $X$ has the strong complementation property.
The ray $\Ray$ is coarsely equivalent to the subspace $\N\subset \Ray$ of natural numbers.
Hence, the Higson corona $\nu \N$ of $\N$ is homeomorphic to $\nu \Ray$
 since coarsely equivalent spaces have homeomorphic Higson coronae.
Thus $\nu \N$ is an indecomposable continuum.
We show that the reverse implication also holds for some class of proper metric spaces.
More precisely, we show that for a non-compact proper metric space $X$
 which is coarsely geodesic and has coarse bounded geometry,
 if the Higson corona $\nu X$ of $X$ is an indecomposable continuum
 then $X$ must be coarsely equivalent to $\N$.
As a result, we give a characterization that, for a non-compact proper metric space $X$
 which is coarsely geodesic and has coarse bounded geometry,
 the Higson corona $\nu X$ of $X$ is an indecomposable continuum if and only if $X$ is coarsely equivalent to $\N$.
Moreover, we characterize a space which is coarsely equivalent to
 the space of integers $\Z$ using the indecomposability of the components of its Higson corona.
Then we apply these characterizations to finitely generated groups.
\par
It is known that a finitely generated group $G$ has 0, 1, 2, or infinitely many ends,
 and the group structure is determined when it has two or infinitely many ends (cf. \cite{BH}, \cite{DK}, \cite{Geoghegan}).
Also, the number of the ends of $G$ is zero
 if and only if $G$ is a finite group.
In that sense, finitely generated groups having exactly one end are fascinating.
We characterize a finitely generated group having exactly one end as a group whose Higson corona is a decomposable continuum.
In contrast, in the case of a group having exactly two ends, we characterize it as a group whose Higson corona is a topological sum of two indecomposable continua, each of which is homeomorphic to $\nu \N$.
\section{Preliminaries}\label{prelim}
Throughout this paper, $\Ray$ denotes the ray $[0, \infty)$ with the metric 
$$d(x,y)=|x-y|$$
 and $\N\subset \Ray$ denotes the space of natural numbers with the induced metric.
In what follows,
 a metric space $(X,d_X)$ is assumed to have a base point $x_0$.

\subsection{Basic properties of the Higson compactification}
In general, the Higson compactification $X^{\nu}$
 of a proper metric space $X$ is defined as the Gelfand dual of
 a unital commutative $C^{\ast}$-algebra of the bounded complex-valued continuous Higson functions on $X$ \cite{Roe}.
On the other hand,
 there is another way to define the Higson compactification of $X$ using an evaluation map of $X$ into Tychonoff cube.
It is known that the compactifications obtained by these two definitions are equivalent
 under the metric coarse structure of $X$.
The latter was first introduced by Keesling \cite{Keesling} and we adopt here the latter definition.
\par
Let $(X, d_X)$ be a metric space and let $B_{d_X}(x,r)$
 be the closed ball of radius $r$ centered at $x\in X$.
A metric $d_X$ on $X$ is called {\it proper} if $B_{d_X}(x, r)$ is compact
 for every $x\in X$ and $r>0$.
For a subset $A$ of $X$, the diameter of $A$ is denoted by $\diam_{d_X} A$,
 that is,
$$\diam_{d_X} A=\sup \{ d_X(x,y)\mid x,y\in A\}.$$

Let $(X,d_X )$ and $(Y,d_Y)$ be proper metric spaces.
A map $f: X\to Y$ is a {\it Higson function}
 provided that
\begin{align}
\tag*{$(\ast)_{r}^f$}
 \qquad \lim_{d_X(x_0, x) \to \infty} \diam_{d_Y} f(B_{d_X}(x,r))=0
\end{align}
 for each $r>0$, that is,
 $f:X\to Y$ is a Higson function
 if and only if, given $r>0$
  and $\varepsilon>0$, there exists a compact subset $K\subset X$
  such that $\diam_{d_Y} f(B_{d_X}(x,r))<\varepsilon$ whenever $x\in X\setminus K$.

Let $C_b(X)$ be the set of all bounded real-valued continuous functions on $X$.
For each $f\in C_b (X)$,
 let $I_f$ denote the closed interval $[\inf f, \sup f]\subset \R$.
For a subset $F$ of $C_b (X)$, let
$$e_F :X\to \prod_{f\in F} I_f$$
 be the {\it evaluation map} of $F$,
 that is,
 $(e_F (x) )_f =f(x)$
  for every $x\in X$.
It is known that
 if $F$ separates points from closed sets,
 then $e_F$ is a topological embedding \cite[2.3.20]{Eng}.
Identifying $X$ with $e_F (X)$,
 the closure $\overline{e_F (X)}$ of $e_F (X)$ in
  $\prod_{f\in F} I_f$
 gives a compactification of $X$.

For a proper metric space $X$,
 we consider the following subsets of $C_b (X)$:
\begin{align*}
C_H(X)&=\{ f\in C_b (X)\mid f \text{ satisfies } (\ast)_{r}^f \text{ for every } r>0 \}.
\end{align*}
Then $C_H(X)$ is a closed subring of $C_b(X)$ with respect to the sup-metric.
Also, it contains all constant maps and separates points from closed sets.
Hence, the subring
 $C_H(X)$ uniquely determines a compactification of $X$ (see \cite[3.12.22 (e)]{Eng})
 which is called the {\it Higson compactification} $X^{\nu}$ of $X$.
Then the compact subset $\nu X=X^{\nu} \setminus X$ is called the {\it Higson corona} of $X$.
\par

The following proposition is a fundamental property of the Higson compactification.

\begin{proposition}
\label{prop:2.1}
\cite[Proposition 1]{Keesling}
Let $X$ be a proper metric space.
Then the Higson compactification is the unique compactification of $X$
 such that if $Y$ is a compact metric space and $f:X\to Y$ is a continuous map,
 then $f$ has a continuous extension $f^{\nu} : X^{\nu} \to Y$
 if and only if
 $f$ satisfies $(\ast )^{f}_{r}$ for every $r>0$.
\end{proposition}

A finite collection $E_1, \cdots , E_n$ of subsets of a metric space $(X,d_X)$
 {\it diverges coarsely} in $X$
if for any $r>0$ the intersection of $r$-neighborhoods of them is bounded,
 that is,
 there exists $R>0$ such that
$$\left(\cap_{i=1}^{n}N_{d_X} (E_i, r)\right)\cap \left(X\setminus B_{d_X} (x_0, R)\right)=\emptyset,$$
 where $N_{d_X}(E_i, r)$ is the $r$-neighborhood of $E_i$ in $X$, i.e.,
$$N_{d_X} (E_i, r)=\{ x\in X\mid d_X (x, E_i )<r\}.$$

For each subset $A$ of a proper metric space $X$,
 $A^{\ast}$ denotes the space $\overline{A}\setminus A$,
 where $\overline{A}$ is the closure of $A$ in the Higson compactification $X^{\nu}$.

\begin{proposition}\label{prop:2.2}
\cite[Proposition 2.3]{DKU}
  Let $X$ be a non-compact proper metric space.
  For a finite collection $E_1,\dots , E_n$ of subsets of $X$,
   the following are equivalent:
  \begin{enumerate}
  \item $\cap_{k=1}^{n} E_k^{\ast} =\emptyset$,
  \item the collection $E_1, \dots, E_n$ diverges coarsely in $X$.
  \end{enumerate}
\end{proposition}

Let $X^{\xi}$ and $X^{\zeta}$ be compactifications of $X$.
We say $X^{\xi} \succeq X^{\zeta}$ provided that
 there is a continuous map $f:X^{\xi} \to X^{\zeta}$ such that $f|_{X}=\mbox{id}_X$.
If $X^{\xi} \preceq  X^{\zeta}$ and $X^{\xi} \succeq X^{\zeta}$
 then we say that $X^{\xi}$ and $X^{\zeta}$ are {\it equivalent compactifications} of $X$.
Of course, two equivalent compactifications of $X$ are homeomorphic.

\begin{proposition}\label{prop:2.3}
\cite[Theorem 1.4]{DKU}
Let $X$ be a non-compact proper metric space
 and let $Y$ be an unbounded closed subspace of $X$ with the induced metric.
Then the closure $\overline{Y}$
 of $Y$ in the Higson compactification $X^{\nu}$ of $X$ is a compactification of $Y$
 which is equivalent to the Higson compactification $Y^{\nu}$ of $Y$.
In particular, $Y^{\ast}$ is homeomorphic to the Higson corona $\nu Y$ of $Y$.
\end{proposition}

\subsection{Maps in large scale geometry and Weighill's property $(\mathfrak{C})$.}
Here, we summarize the basic properties of maps used in large scale geometry.
Then we state Weighill's property $(\mathfrak{C})$ and related results.
\par
A (not necessarily continuous) function $f:X\to Y$ between metric spaces $X$ and $Y$
 is called a {\it coarse map}
 if $f$ satisfies the following two conditions:
\begin{enumerate}
\item[$\bullet$] $f$ is {\it uniformly expansive}, that is,
 there exists a non-decreasing function $\sigma : \Ray \to \Ray$
 such that
 $$
 d_Y (f(x), f(x'))\leq \sigma (d_X (x, x'))
 $$
 for every $x, x' \in X$
 (in this case, $f$ is said to be {\it $\sigma$-uniformly expansive}), and
\item[$\bullet$] $f$ is {\it metrically proper},
 that is,
 for any bounded subset $B\subset Y$,
 $f^{-1}(B)$ is bounded in $X$ with respect to the metric $d_X$.
\end{enumerate}
Two maps $f:X\to Y$ and $g:X\to Y$ are said to be {\it close}
 if there exists $r>0$ such that
 $d_Y (f(x), g(x))<r$
 for every $x\in X$.
A coarse map $f:X\to Y$ is called a {\it coarse equivalence}
 if there exists a coarse map $g:Y\to X$ such that
 $g\circ f$ and $f\circ g$ are close to $\mbox{id}_X$ and $\mbox{id}_Y$ respectively.
If there exists a coarse equivalence between $X$ and $Y$
 then $X$ and $Y$ are said to be {\it coarsely equivalent}.

A (not necessarily continuous) function $f:X\to Y$ between metric spaces $X$ and $Y$
 is said to be {\it uniformly metrically proper}
 if the following condition is satisfied:
\begin{enumerate}
 \setcounter{enumi}{2}
\item[$\bullet$] there exists a non-decreasing function $\tau : \Ray \to \Ray$
 such that 
 $$
 d_X (x, x')\leq \tau (d_Y (f(x), f(x')))
 $$
 for every $x, x' \in X$ (in this case, $f$ is said to be {\it $\tau$-uniformly metrically proper}).

\end{enumerate}

\begin{proposition}\label{prop:2.4}
Let $X$ and $Y$ be unbounded metric spaces.
Then a map $f:X\to Y$ is uniformly metrically proper if and only if
 there exists a non-decreasing function $\rho_{-}: \Ray \to \Ray$
 such that
 \begin{enumerate}
 \item $\displaystyle\lim_{t\to \infty} \rho_{-}(t)=\infty$,
 and
 \item $ \rho_{-}(d_X (x,x'))\leq d_Y (f(x), f(x')) $
 holds for every $x, x' \in X$.
 \end{enumerate}
\end{proposition}

\begin{proof}
Suppose that $f$ is uniformly metrically proper.
Then there exists a non-decreasing function $\tau : \Ray \to \Ray$
 such that 
\begin{align}
\tag*{$(2.1)$}\label{2.4-1}
 d_X (x, x')\leq \tau (d_Y (f(x), f(x')))
\end{align}
 for every $x, x' \in X$.
Since $X$ is unbounded,
 \ref{2.4-1} implies that $\displaystyle\lim_{t\to \infty} \tau(t)=\infty$.
Thus we can take a sequence $\{ n(i)\}_{i=1}^{\infty}\subset \N$
 such that
 $n(i)<n(i+1)$ and $0<\tau (n(i)) < \tau (n(i+1))$ for every $i\in \N$.
Then we define $\rho_{-}:\Ray \to \Ray$ by
\begin{equation*}
\rho_{-}(t)
 =
 \begin{cases}
   0 & \mbox{if}\ 0\leq t\leq \tau(n(1)),\\
   n(i)  & \mbox{if}\ \tau(n(i))<t\leq \tau(n(i+1)).
 \end{cases}
\end{equation*}
Then $\rho_{-}$ is a non-decreasing function with
 $\displaystyle\lim_{t\to \infty} \rho_{-}(t)=\infty$.
To check (2),
 we note that \ref{2.4-1} is equivalent to the condition that
\begin{align}
\tag*{$(2.2)$}\label{2.4-2}
d_X (x, x')>\tau (t) \Longrightarrow d_Y (f(x), f(x'))>t
\end{align}
 for every $x, x'\in X$.
Thus, if $\tau (n(i))<d_X (x,x')\leq \tau(n(i+1))$
 then $$d_Y (f(x), f(x'))>n(i)=\rho_{-}(d_X (x, x'))$$ by \ref{2.4-2}.
Thus, $(2)$ is satisfied.
 \par
Next, let $\rho_{-}: \Ray \to \Ray$ be a a non-decreasing function satisfying
 (1) and (2).
Then we can take a sequence $\{ m(i)\}_{i=1}^{\infty}\subset \N$
 such that
 $m(i)<m(i+1)$ and $0<\rho_{-} (m(i)) < \rho_{-} (m(i+1))$ for every $i\in \N$.
Note that (2) is equivalent to the condition that
\begin{align}
\tag*{$(2.3)$}\label{2.4-3}
d_Y (f(x), f(x'))<\rho_{-}(t) \Longrightarrow d_X (x, x')<t
\end{align}
 for every $x, x'\in X$.
Now we define $\tau :\Ray\to \Ray$ by
\begin{equation*}
\tau(t)
 =
 \begin{cases}
   m(1) & \mbox{if}\ 0\leq t<\rho_{-}(m(1)),\\
   m(i)  & \mbox{if}\  \rho_{-}(m(i-1))\leq t<\rho_{-}(m(i)).
 \end{cases}
\end{equation*}
Clearly, $\tau$ is non-decreasing.
If $\rho_{-}(m(i-1))\leq d_Y (f(x), f(y))<\rho_{-}(m(i))$
 then $$d_X (x, x')<m(i)=\tau(d_Y (f(x), f(x')))$$
 by \ref{2.4-3}.
Thus $f$ is uniformly metrically proper.
\end{proof}

\begin{proposition}\label{prop:2.5}
\cite[Proposition 3.8]{ALC1}
Any coarse equivalence between metric spaces is uniformly metrically proper.
\end{proposition}

A map $f$ is called a {\it rough map}
 if it is uniformly expansive and uniformly metrically proper.
If a rough map $f$ is $\sigma$-uniformly expansive and $\tau$-uniformly metrically proper
 then $f$ is called a {\it $(\sigma, \tau)$-rough map}.
In case $\sigma =\tau$, a $(\sigma, \tau)$-rough map $f$ is called a {\it $\sigma$-rough map}.
Note that if we define $\lambda=\max\{\sigma, \tau\}$
 then any $(\sigma, \tau)$-rough map is a $\lambda$-rough map.
A rough map $f:X\to Y$ is called a {\it rough equivalence}
 if there exists a rough map $g:Y\to X$ such that
 $g\circ f$ and $f\circ g$ are close to $\mbox{id}_X$ and $\mbox{id}_Y$ respectively.
A rough equivalence which is a $\sigma$-rough map is called a {\it $\sigma$-rough equivalence}.
\par
Let $r>0$.
A subset $Y$ of $X$ is called an {\it $r$-net} in $X$
 if for every $x\in X$, there is $y\in Y$
 such that $d_X (x, y)<r$.
A {\it net} of $X$ is a subset of $X$ which is an $r$-net for some $r>0$.

\begin{proposition}\label{prop:2.6}
\cite[Proposition 2.37]{ALC0}
If a map $f:X\to Y$ between metric spaces
 is a rough map and $f(X)$ is a net in $Y$
 then $f$ is a rough equivalence.
\end{proposition}

Thus, a net $A$ in a space $X$ is coarsely equivalent to $X$.

\begin{corollary}\label{cor:2.7}
\cite[Corollary 2.38]{ALC0}
For a map $f:X\to Y$ between metric spaces,
 the following conditions are equivalent:
\begin{enumerate}
\item $f$ is a rough map,
\item $f: X\to f(X)$ is a rough equivalence,
\item $f: X\to f(X)$ is a coarse equivalence.
\end{enumerate}
\end{corollary}

\begin{proposition}\label{prop:2.8}
(cf. \cite[Definition 1.4.4]{NY})
Let $X$ and $Y$ be unbounded metric spaces.
Then a map $f:X\to Y$
is a coarse equivalence if and only if
 $f$ satisfies the following two conditions:
\begin{enumerate}
\item $f(X)$ is a net in $Y$, and
\item  there exist non-decreasing functions $\rho_{+}$, $\rho_{-}: \Ray \to \Ray$
 such that 
 $$\displaystyle\lim_{t\to \infty} \rho_{-}(t)=\infty$$
 and the inequality
 $$
 \rho_{-}(d_X (x,x'))\leq d_Y (f(x), f(x')) \leq \rho_{+}(d_{X}(x, x'))
 $$
 holds for every $x, x' \in X$.
\end{enumerate}
\end{proposition}

\begin{proof}
If $f$ is a coarse equivalence
 then it is easy to see that $f(X)$ is a net in $Y$.
By Corollary \ref{cor:2.7}, $f$ is also a rough equivalence.
Say $f$ is a $\sigma$-rough equivalence.
Then $\rho_{+}=\sigma$ and the existence of $\rho_{-}$ is assured by Proposition \ref{prop:2.4}.
\par
Conversely, suppose that $f$ satisfies $(1)$ and $(2)$.
Then $f$ is $\rho_{+}$-uniformly expansive.
Also, it is uniformly metrically proper by Proposition \ref{prop:2.4}.
Thus $f$ is a coarse equivalence by Proposition \ref{prop:2.6}
 and Corollary \ref{cor:2.7}.
\end{proof}

The following was originally stated in \cite[ Proposition 2.41 and  Corollary 2.42]{Roe}
 for coarse maps between proper coarse spaces.
See \cite[Corollary 4.14, Proposition 4.15, and Theorem 4.16]{ALC1}
 for coarse maps between proper metric spaces.

\begin{proposition}\label{prop:2.9}
\cite{Roe} (cf. \cite{ALC1})
Let $X$ and $Y$ be proper metric spaces.
A coarse map $f: X\to Y$ extends to a  continuous map $\nu f: \nu X\to \nu Y$
 between their Higson coronae.
If $f, g : X\to Y$ are close coarse maps then $\nu f=\nu g$.
In particular, if $f: X\to Y$ is a coarse equivalence
 then $\nu f$ is a homeomorphism between their Higson coronae.
\end{proposition}

Let $(Y, d_Y)$ and $(Z, d_Z)$ be metric spaces with base points $y_0 \in Y$ and $z_0 \in Z$
 respectively.
The {\it coarse coproduct} of $(Y, d_Y)$ and $(Z, d_Z)$
 is the metric space $(Y+Z, d_{Y+Z})$ whose underlying set is the disjoint union of $Y$ and $Z$
 and the distance function $d_{Y+Z}$ is defined as follows:
\begin{equation*}
 d_{Y+Z}(a,b)
 =
 \begin{cases}
   d_Y (a,b) & \mbox{if}\ a,b\in Y,\\
   d_Z (a,b) & \mbox{if}\ a,b\in Z,\\
   d_Y (y_0, a)+1+d_Z (z_0, b) & \mbox{if}\ a \in Y,\  b\in Z,\\
   d_Z (z_0, a)+1+d_Y (y_0, b)  & \mbox{if}\ a\in Z,\  b\in Y.\  
 \end{cases}
\end{equation*}

For a metric space $X$, we consider the following property defined by Weighill \cite{Wei}:
\begin{list}{}{}
\item[$(\mathfrak{C})$]
  for every coarse map $f:X\to Y+Z$,
  we can take $A\in \{Y, Z\}$
  and a coarse map $g_A : X\to A$ such that $i_A \circ g_A : X\to Y+Z$ is close to $f$,
  where $i_A:A\to Y+Z$ is the inclusion.
\end{list}

Although the Higson compactification is defined for proper metric spaces,
 the Higson corona can be defined for arbitrary metric spaces.
Then Weighill proved the following:

\begin{theorem}\label{thm:2.10}
\cite[Theorem 4.6]{Wei}
Let $X$ be a non-compact metric space.
The Higson corona $\nu X$ of $X$ is topologically connected
 if and only if $X$ satisfies $(\mathfrak{C})$.
\end{theorem}

The following is a consequence of \cite[Theorem 4.1]{Wei}, \cite[Proposition 4.5]{Wei},
 and Proposition \ref{prop:2.3}.

\begin{proposition}\label{prop:2.11}
If a proper metric space $X$ does not satisfy $(\mathfrak{C})$
 then there are two unbounded closed sets $A$ and $B$ of $X$ such that
\begin{enumerate}
\item $X=A\cup B$,
\item $A$ and $B$ diverges coarsely in $X$, and
\item the Higson corona $\nu X$ is homeomorphic to
 the topological sum of the Higson coronae of $A$ and $B$,
 i.e., $\nu X \approx \nu A \oplus \nu B$.
\end{enumerate}
\end{proposition}


\section{Indecomposable continua as Higson coronae}\label{indec}
In this section,
 we consider spaces whose Higson coronae are indecomposable continua.
\par

A {\it continuum} is a non-empty, compact and topologically connected Hausdorff space.
A {\it subcontinuum} is a continuum which is a subset of a continuum.
A {\it proper} subset of $X$ is a subset of $X$ which is not equal to $X$.
A continuum is called {\it decomposable}
 if it can be represented as the union of two of its proper subcontinua.
A continuum which is not decomposable is said to be {\it indecomposable}.
\par
An example of a space whose Higson corona is an indecomposable continuum is the following:

\begin{theorem}\label{thm:3.1}
\cite[Theorem 1.6]{IT}
The Higson corona $\nu\Ray$ is a non-metrizable indecomposable continuum.
\end{theorem}

Since the inclusion $\N \to \Ray$ is a coarse equivalence,
 $\nu \N$ is homeomorphic to $\nu \Ray$.
Hence, we have the following:

\begin{corollary}\label{cor:3.2}
The Higson corona $\nu\N$ is a non-metrizable indecomposable continuum
 which is homeomorphic to $\nu \Ray$.
\end{corollary}

\begin{remark}
The subpower Higson compactification was introduced in \cite{KZ1}
 as a variant of the Higson compactification.
It is known that the subpower Higson corona of $\Ray$ is also a non-metrizable indecomposable continuum
 \cite{Iwa}.
\end{remark}

Let $\mu>0$ and let $X$ be a metric space.
A subset $Y$ of $X$ is said to be {\it $\mu$-connected}
 if, for every two points $x$ and $y$ of $Y$,
  there exists a finite sequence $\{p_i\}_{i=1}^{n}$ in $Y$ such that
 $p_1 =x$, $p_n =y$ and $d_X(p_{i}, p_{i+1})\leq\mu$ for every $i$.

\begin{lemma}\label{lem:3.4}
Let $(X, d_X)$ be a non-compact proper metric space.
If there exists a coarse equivalence $f: M\to X$ from a geodesic metric space $M$
 then there exist a positive number $\mu>0$ and a non-decreasing function $\tau: \Ray \to \Ray$
 with $\displaystyle\lim_{t\to \infty}\tau(t)=\infty$ such that,
 for each unbounded sequence $\{y_n \}_{n=1}^{\infty}\subset f(M)$,
 there exist a subsequence $\{x_n \}_{n=1}^{\infty}\subset \{y_n \}_{n=1}^{\infty}$
 and a family $\{B_n \}_{n=1}^{\infty}$ of subsets of $X$ satisfying the following:
\begin{enumerate}
\setcounter{enumi}{0}
\item[$(\mbox{A})$] $\diam B_n \leq \tau (n)$,
\item[$(\mbox{B})$] $B_n$ is $\mu$-connected,
\item[$(\mbox{C})$] $B_n$ contains the $n$-ball $B_{d_X}(x_n, n)$
 for every $n\in \N$, and
\item[$(\mbox{D})$] $\displaystyle\lim_{n\to \infty} d_X(x_0, B_n)=\infty$.
\end{enumerate}
\end{lemma}

\begin{proof}
Let $f: M\to X$ be a coarse equivalence from a geodesic metric space $(M, d_M)$.
Since any coarse equivalence is a rough equivalence,
 we may assume that $f$ is a $\sigma$-rough equivalence
 for some non-decreasing function $\sigma :\Ray \to \Ray$,
 that is,
\begin{enumerate}
\item\label{r-1} $d_M (x,y)\leq \sigma(d_X(f(x), f(y)))$, and
\item\label{r-2} $d_X (f(x), f(y))\leq \sigma (d_M (x, y))$
\end{enumerate}
for every $x, y\in M$.
Note that $\displaystyle\lim_{t\to \infty}\sigma(t)=\infty$
 since $X$ is a non-compact proper metric space.
By Proposition \ref{prop:2.8},
 we may assume that $f(M)$ is an $r$-net in $X$ for some $r>0$.
Let $\{y_n \}_{n=1}^{\infty}\subset f(M)$ be an unbounded sequence.
Then we can take a subsequence $\{ x_n \}_{n=1}^{\infty}$ of $\{y_n \}_{n=1}^{\infty}$
 such that
\begin{enumerate}
\setcounter{enumi}{2}
\item\label{r-3}
 $d_X(x_n, x_0)>n+\sigma^{2}(n+r)$
\end{enumerate}
 for each $n$, where $\sigma^{2}(n+r)=\sigma (\sigma (n+r))$.
Let $\{ v_n \}_{n=1}^{\infty}$ be a sequence in $M$ such that
$$
f(v_n)=x_n
$$
for every $n$.
Put 
$$D_n =B_{d_M}(v_n, \sigma(n+r))$$
 for each $n\in \N$.
Then we define $B_n$ as the $r$-neighborhood of $f(D_n)$, that is,
 $$
 B_n =N_{d_X}(f(D_n) , r).
 $$
Define $\tau : \Ray \to \Ray$ by
 $$ \tau(t)=2(\sigma^2 (t+r)+r). $$
Note that $\tau$ is a non-decreasing function with $\displaystyle\lim_{t\to \infty}\tau(t)=\infty$.
Then we have 
\begin{align*}
\diam B_n &\leq \diam f(D_n) +2r \\
&\leq 2\sigma (\sigma (n+r))+2r=\tau (n).
\end{align*}
Thus, $(\mbox{A})$ is satisfied.
\par

Put $\mu=\sigma(1)+r.$
Since $M$ is a geodesic metric space,
 each ball $D_n$ is $1$-connected.
So, $B_n$ is $\mu$-connected since $f(D_n)$ is $\sigma(1)$-connected.
Thus, $(\mbox{B})$ is satisfied.
\par
To check $(C)$,
 let $x\in B_{d_X}(x_n, n)$.
Since $f(M)$ is an $r$-net in $X$,
 we can take $y=f(m)\in f(M)$, $m\in M$, such that $d_X (x, y)<r$.
Then
$$
 d_X (f(v_n), f(m))=d_X(x_n, y)
 \leq d_X (x_n, x)+d_X (x, y)< n+r.
$$
By $(\ref{r-1})$, we have $d_M (v_n, m)\leq \sigma (n+r)$.
So we have
$$y=f(m)\in f(B_{d_M}(v_n, \sigma(n+r)))=f(D_n),$$
 that is,
 $x\in N_{d_X}(f(D_n) , r)=B_n$.
Thus, $(C)$ is satisfied.
\par
Finally, we shall check $(\mbox{D})$.
Let $x\in B_n$.
We take $y=f(m)\in f(D_n)$, $m\in D_n\subset M$,
 such that $d(x, y)<r$.
Then we have
 \begin{align*}
 d_X(x, x_n)&\leq d_X (x,y)+d_X (y, x_n)= d_X (x,y)+d_X(f(m), f(v_n))\\
 &< r+\sigma(d_M (m, v_n))\leq  r+\sigma (\sigma (n+r))=r +\sigma^2 (n+r).
 \end{align*}
Using $(\ref{r-3})$, we have
 \begin{align*}
 d_X(x, x_0)&\geq d_X(x_n, x_0 )-d_X(x, x_n)\\
 &> n+\sigma^2 (n+r)-(r+\sigma^2(n+r))= n-r.
 \end{align*}
So we have $\displaystyle\lim_{n\to \infty} d_X(x_0, B_n)=\infty$,
 $(\mbox{D})$ is satisfied.
\end{proof}

\begin{lemma}\label{lem:3.5}
Let $\mu>0$.
Let $(X, d_X)$ be a non-compact proper metric space and
 let $\Gamma $ be a closed subset of $X$ such that $\Gamma^{\ast}$ is connected.
Suppose that $\{B_n\}_{n=1}^{\infty}$ is a family of $\mu$-connected bounded subsets of $X$
 such that
\begin{enumerate}
\item[(i)] $\Gamma \cap B_n \neq \emptyset$ for every $n$, and
\item[(ii)] $\displaystyle\lim_{n\to \infty}d_X(x_0, B_n)=\infty.$
\end{enumerate}
Let $E=\Gamma\cup (\cup_{n=1}^{\infty} B_n)$.
Then $E^{\ast}$ is a continuum.
\end{lemma}

\begin{proof}
To see the connectivity of $E^{\ast}$,
 it suffices to show that $E$ satisfies $(\mathfrak{C})$
 by Proposition \ref{prop:2.3} and Theorem \ref{thm:2.10}.
Let $(Y, d_Y)$ and $(Z, d_Z)$ be metric spaces with base points $y_0 \in Y$ and $z_0 \in Z$ respectively.
Let $f: E \to Y+Z$ be a coarse map into a coarse coproduct space $Y+Z$.
We may assume that $f$ is $\sigma$-uniformly expansive for some non-decreasing function
 $\sigma: \Ray \to \Ray$, i.e., 
 $$d_{Y+Z}(f(x), f(y))\leq \sigma(d_X (x,y))$$
 for every $x, y\in E$.
Since $\Gamma^{\ast}$ is connected,
 using Theorem \ref{thm:2.10},
 we may assume without loss of generality that
 there is a coarse map 
 $$g: \Gamma\to Y$$
 such that $i_Y\circ g$ and $f|_\Gamma$ are $r_0$-close
 for some $r_0>0$,
 where $i_Y: Y\to Y+Z$ is the inclusion.
Since $g$ is metrically proper,
 there exists $r_1>0$ such that
 $$g^{-1}(B_{d_Y}(y_0, r_0))\subset  B_{d_X}(x_0, r_1).$$
Then $d_Y (g (x), y_0)>r_0$ for every $x\in \Gamma\setminus B_{d_X}(x_0 , r_1)$.
So we have
\begin{enumerate}
\setcounter{enumi}{0}
\item\label{cc-1}
 $f(x)\in Y$ for every $x\in \Gamma\setminus B_{d_X}(x_0, r_1)$.
\end{enumerate}
Indeed, if there exists a point $x\in \Gamma\setminus B_{d_X}(x_0, r_1)$ such that $f(x)\in Z$
 then
\begin{align*}
d_{Y+Z}(i_Y \circ g (x), f(x))
&=d_Y (g(x), y_0)+1+d_Z (f(x), z_0)\\
&> r_0 +1 >r_0.
\end{align*}
This contradicts the fact that $i_Y\circ g$ and $f|_\Gamma$ are $r_0$-close.
\par

Since $f$ is metrically proper, we can take $r_2 >r_1$ so that
\begin{enumerate}
\setcounter{enumi}{1}
\item\label{cc-2}
 $f^{-1} (B_{d_{Y+Z}}(y_0, \sigma(\mu)))\subset B_{d_X}(x_0, r_2)$.
\end{enumerate}
Put
 $$\Lambda =\{ k \in \N\mid  B_k\cap B_{d_X}(x_0, r_2)\neq\emptyset\}.$$
Then $\Lambda$ is a finite set since $\displaystyle \lim_{n\to \infty}d_X(x_0, B_n)=\infty$.
Put
 $$B=\left(E\cap  B_{d_X}(x_0, r_2)\right) \cup (\cup_{k\in \Lambda}B_k) .$$
Note that $B$ is a bounded set in $E$ since $\Lambda$ is finite.
Then it follows that
\begin{enumerate}
\setcounter{enumi}{2}
\item\label{cc-3}
 $f(E\setminus B)\subset Y$.
\end{enumerate}
Indeed, if $x\in \Gamma\setminus B\subset E\setminus B$ then $f(x)\in Y$ by $(1)$.
Suppose $x\in E\setminus (\Gamma\cup B)$.
Then there exists $k\in \N \setminus \Lambda$
 such that $x\in B_k\subset E\setminus B_{d_X}(x_0, r_2)$.
Since $\Gamma\cap B_k\neq \emptyset$,
 we can take a point $b_k \in \Gamma\cap B_k$.
Note that $f(b_k )\in Y$ by $(1)$.
Since $B_k$ is $\mu$-connected,
 there is a finite sequence $\{ p_i\}_{i=1}^{n}\subset B_k$
 such that $p_1=b_k$, $p_n =x$ and $d_{X}(p_{i-1}, p_{i})\leq\mu$ for every $i$.
Then
 $$d_{Y+Z}(f(p_i), f(p_{i+1})) \leq \sigma (\mu)$$
 for every $i$.
If $f(x)$ is contained in $Z$
 then
 there is $1<j\leq n$ such that $f(p_j)\in Z$ and $f(p_{j-1})\in Y$
 since $f(p_1)=f(b_k)\in Y$ and $f(p_n )=f(x)\in Z$.
Since $p_{j-1} \not\in B_{d_X}(x_0, r_2)$,
 we have
 $d_{Y}(f(p_{j-1}),y_0)>\sigma(\mu)$
 by $(2)$.
So we have
\begin{align*}
d_{Y+Z}(f(p_{j-1}), f(p_{j}))
& =d_{Y}(f(p_{j-1}),y_0)+1+d_{Z}(f(p_{j}), z_0)\\
& >\sigma(\mu)+1>\sigma (\mu)\geq \sigma(d_X (p_{j-1}, p_j)),
\end{align*}
 contradicting the $\sigma$-uniform expansiveness of $f$.
Hence,
 $f(x)\in Y$.
\par

Now we define $h:E \to Y$
 by $h|_{E\setminus B}=f|_{E\setminus B}$ and $h(x)=y_0$ for every $x\in B$.
The function $h$ is well-defined by $(3)$.
Since $B$ is a bounded set,
 $h$ is metrically proper and close to $f$.
To see that $h$ is uniformly expansive,
 we fix $b_0 \in B$
 and define a function $\tau : \Ray \to \Ray$
  by 
  $$\tau(t)= \sigma(t) +d_Y (y_0, f(b_0))+\sigma (\diam B)$$
  for every $t\in \Ray$.
Clearly $\tau$ is a non-decreasing function such that
 $\sigma(t)\leq \tau(t)$ for every $t\in \Ray$.
Then for $x\in B$ and $x'\in E\setminus B$,
 we have
\begin{align*}
d_Y (h(x), h(x'))&=d_Y(y_0, f(x'))\\
&\leq d_Y (y_0, f(b_0))+d_Y (f(b_0), f(x))+d_Y (f(x), f(x'))\\
&\leq d_Y (y_0, f(b_0))+\sigma (\diam B)+\sigma(d_X (x,x'))\\
&=\tau(d_X (x,x')).
\end{align*}
This establishes that $h$ is $\tau$-uniformly expansive.
Thus $h$ is a coarse map close to $f$.
Hence, $E$ satisfies $(\mathfrak{C})$.
\end{proof}

The next lemma follows from the proof of \cite[Theorem 3.11]{ALC1} since any geodesic metric space is coarsely quasi-convex
 and any coarse quasi-isometry induces a coarse equivalence (cf. \cite[Proposition 3.25]{ALC0}).

\begin{lemma}\label{lem:3.6}
\cite{ALC1}
If $X$ is a geodesic metric space
 then there exist
 a connected graph $G$
 and a coarse equivalence $f:V(G)\to X$ from the vertices $V(G)$ of $G$,
 where the metric on $V(G)$ is induced by the path metric on $G$
 such that each edge of which is isometric to the unit interval.
\end{lemma}

A space $X$ is said to have {\it coarse bounded geometry} if there exist a net $A$ of $X$ and
 a function $\lambda: \Ray \to \Ray$
  such that $\#(A \cap B(x, t))\leq \lambda(t)$ for every $x\in X$ and $t>0$,
  where $\#(A \cap B(x, t))$ denotes the number of elements of $A \cap B(x, t)$.

It is known that if $X$ is the metric space of vertices of a connected graph $G$
 then $X$ has coarse bounded geometry if and only if $G$ is of {\it finite type},
 that is, there is a positive integer $K$ such that every vertex of $G$ has degree bounded by $K$
 (cf. \cite[Example 3.7]{ALC0}).

\begin{proposition}[\cite{ALC0}, Proposition 3.10]\label{prop:3.7}
Let $X$ and $Y$ be metric space.
If $X$ has coarse bounded geometry and 
 $Y$ is coarsely equivalent to $X$ then $Y$ has coarse bounded geometry.
\end{proposition}

A metric space $X$ is said to be {\it coarsely geodesic}\footnote{
The coarse structure of a coarsely geodesic space is called monogenic in \cite{Roe}.
See \cite[Proposition 2.57]{Roe}.}
 if $X$ is coarsely equivalent to a geodesic metric space.

\begin{lemma}\label{lem:3.8}
Let $X$ be a non-compact proper metric space.
If $X$ is coarsely geodesic and has coarse bounded geometry
 then there exists a rough map $\alpha : \N \to X$.
\end{lemma}

\begin{proof}
Since $X$ is coarsely geodesic,
 there exists a coarse equivalence
 $f: M\to X$ from a geodesic metric space $M$.
By Lemma \ref{lem:3.6},
 there exist a connected graph $G$
 and a coarse equivalence $g:V(G)\to M$ from the vertices $V(G)$ of $G$.
Note that $G$ must be an infinite graph since $X$ is a non-compact proper metric space.
By Proposition \ref{prop:3.7}, $G$ must be of finite type since $M$ has coarse bounded geometry.
It is known \cite[8.2.1]{Die} that
 every infinite connected graph has a vertex of infinite degree or contains a ray.
Since $G$ is of finite type, there exists a ray $R\subset G$.
Note that the set of vertices $V(R)$ of $R$ is isometric to $\N$.
So let $h: \N \to V(R)\subset V(G)$ be the isometry.
Then the composition $\alpha =f\circ g\circ h : \N \to X$
 is a required rough map
 since $f$ and $g$ are rough equivalences by Corollary \ref{cor:2.7}.
\end{proof}

The following proposition is useful to verify the decomposability of a continuum.

\begin{proposition}\label{prop:3.9}
\cite[Theorem 3.41]{HY}
A continuum $K$ is decomposable if and only if
 there exists a proper subcontinuum $C\subset K$ with non-empty interior in $K$.
\end{proposition}

Now we shall give a characterization of a space
 which is coarsely equivalent to $\N$
  using the indecomposability of its Higson corona.

\begin{theorem}\label{thm:3.10}
Let $X$ be a non-compact proper metric space.
Then $X$ is coarsely equivalent to $\N$ if and only if
 the following two conditions are satisfied:
\begin{enumerate}
\renewcommand{\labelenumi}{(\roman{enumi})}
\item $X$ is coarsely geodesic and has coarse bounded geometry, and
\item the Higson corona $\nu X$ is an indecomposable continuum.
\end{enumerate}
\end{theorem}

\begin{proof}
Suppose that $X$ is coarsely equivalent to $\N$.
Since $\N$ is coarsely equivalent to $\R_{+}$,
 $X$ satisfies (i) and (ii) by
 Theorem \ref{thm:3.1} and Proposition \ref{prop:3.7}.
\par

Now, let $X$ be a non-compact proper metric space satisfying (i) and (ii).
Assume that $X$ is not coarsely equivalent to $\N$.
Then we shall show that $\nu X$ is a decomposable continuum.
Let $\alpha :\N \to X$ be a rough map assured by Lemma \ref{lem:3.8}.
Put $$\Gamma=\alpha(\N).$$
Note that $\Gamma$ is an unbounded closed subset of $X$ since $\alpha$ is a rough map.
By Proposition \ref{prop:2.3},
 $\Gamma^{\ast}$ is homeomorphic to $\nu \Gamma$ which is homeomorphic to $\nu \N$.
Hence, $\Gamma^{\ast}$ is a continuum by Theorem \ref{thm:3.1}.
\par
Since $X$ is coarsely geodesic,
 there exists a coarse equivalence $f: M\to X$ from a geodesic metric space $M$.
Then we take a positive number $\mu>0$
 and a non-decreasing function $\tau :\Ray\to \Ray$
 with $\displaystyle\lim_{t\to \infty}\tau(t)=\infty$ as in Lemma \ref{lem:3.4} for $f: M\to X$.
Since $X$ is not coarsely equivalent to $\N$,
 the rough map $\alpha: \N\to X$ cannot be a coarse equivalence.
Hence, $\Gamma$ cannot be a net in $X$ by Proposition \ref{prop:2.6}.
So we can take a sequence $\{ z_i \}_{i=1}^{\infty}\subset X$
 such that
\begin{enumerate}
\item\label{c-1}
 $d_X(z_n, \Gamma)>3\tau(n)$ for each $n\in \N$.
\end{enumerate}
Put $Z=\{ z_i \}_{i=1}^{\infty}$.
Note that $Z^{\ast}\neq\emptyset$ since $Z$ is an unbounded set.
Since $\Gamma$ is unbounded,
 we can take an unbounded sequence $\{ y_i\}_{i=1}^{\infty}\subset \Gamma$,
 such that
\begin{enumerate}
 \setcounter{enumi}{1}
\item\label{c-2}
 $d_X(z_k, y_{n})> 3 \tau(n)$ for every $1\leq k\leq n$.
\end{enumerate}
By Lemma \ref{lem:3.4},
 there exist a subsequence $\{x_n \}_{n=1}^{\infty}\subset \{y_n \}_{n=1}^{\infty}\subset \Gamma$
 and a family $\{B_n \}_{n=1}^{\infty}$ of subsets of $X$ satisfying the following:
\begin{enumerate}
\setcounter{enumi}{0}
\item[$(\mbox{A})$] $\diam B_n <\tau(n)$,
\item[$(\mbox{B})$] $B_n$ is $\mu$-connected,
\item[$(\mbox{C})$] $B_n$ contains the $n$-ball $B(x_n, n)$
 for every $n\in\N$, and
\item[$(\mbox{D})$] $\displaystyle\lim_{n\to \infty} d_X(x_0, B_n)=\infty$.
\end{enumerate}
By (\ref{c-2}), we may assume that
\begin{enumerate}
\setcounter{enumi}{2}
\item\label{c-3}
 $d_X(z_k, x_{n})> 3 \tau(n)$ for every $1\leq k\leq n$.
\end{enumerate}
Put
 $$E=\Gamma \cup \left( \cup_{n=1}^{\infty}B_n\right).$$
Then $E^{\ast}$ is a continuum by Lemma \ref{lem:3.5}.
\par
Now we shall show that $Z$ and $E$ diverge coarsely in $X$.
Given $r>0$,
 we can take $n_0 >0$ such that $\tau(n_0)>r$ since $\displaystyle\lim_{t\to \infty}\tau(t)=\infty$.
Fix $m\geq n_0$.
Then $d_X(z_m, \Gamma)>3\tau(m)\geq 3\tau(n_0)>3r$ by $(\ref{c-1})$.
So we have $B(z_m, r)\cap N(\Gamma,r)=\emptyset$.
Recall that $x_n \in B_n \cap \Gamma$ for each $n\in \N$.
So, if $n< m$ then
\begin{align*}
 d_X(z_m, B_n)
 &\geq d_X(z_m, \Gamma)-\diam B_n\\
 &>3\tau(m) -\tau(n) \geq 2\tau(m)>2r
\end{align*}
  by $(\ref{c-1})$ and $(\mbox{A})$.
If $m\leq n$
 then
\begin{align*}
 d_X(z_m, B_n)
 &\geq d_X(z_m, x_n)-\diam B_n\\
 &>3\tau(n) -\tau(n) =2\tau(n)>2r
\end{align*}
 by $(\ref{c-3})$ and $(\mbox{A})$.
As a consequence, $B(z_m, r)\cap N(B_n ,r)=\emptyset$ for every $n$.
Thus the intersection $N(Z, r)\cap N(E, r)$
 is contained in the bounded set $\cup_{i=1}^{n_0 -1} B(z_i, r)$,
 that is, the collection $Z$ and $E$ diverge coarsely in $X$.
Hence, $Z^{\ast}\cap E^{\ast}=\emptyset$ by Proposition \ref{prop:2.2}.
Since $Z^{\ast}\neq \emptyset$,
 this establishes that $E^{\ast}$ is a proper subcontinuum of $\nu X$.
\par

By $(\mbox{C})$, $E$ contains $B(x_n, n)$ for every $n\in \N$,
 that is,
 $E$ contains disjoint metric balls of arbitrary large radius.
Hence, $E^{\ast}$ has non-empty interior\footnote{
 In fact, one can construct a Higson function
 $h: X\to [0,1]$ such that $\mbox{supp}\, h\subset \cup_{n=1}^{\infty} B(x_n, n)\subset E$
 and $h(x_n )=1$ for infinitely many $n$.
By Proposition \ref{prop:2.1}, there is the extension $\hat{h}: X^{\nu}\to [0,1]$ of $h$.
Then
 $\hat{h}^{-1}((0,1])$ is open in $X^{\nu}$
  and $\hat{h}^{-1}((0,1])\cap E^{\ast}\neq \emptyset$ by $(\mbox{D})$.
In particular, $\hat{h}^{-1}((0,1])\cap \nu X \subset E^{\ast}$.
 }
 in $\nu X$ \cite[Proposition 4.12]{ALC1}.
We have shown that $E^{\ast}$ is a proper subcontinuum of $\nu X$ with non-empty interior.
By Proposition \ref{prop:3.9}, $\nu X$ must be a decomposable continuum, a contradiction.
\end{proof}

\begin{corollary}\label{cor:3.11}
Suppose that $X$ is a non-compact proper metric space
 which is coarsely geodesic and has coarse bounded geometry.
Then the Higson corona $\nu X$ of $X$ is an indecomposable continuum if and only if $X$ is coarsely equivalent to $\N$.
\end{corollary}

Next we shall give a characterization of a space
 which is coarsely equivalent to $\Z$
  using the indecomposability of the components of its Higson corona.
\par

\begin{theorem}\label{thm:3.12}
Let $X$ be a non-compact proper metric space.
Then $X$ is coarsely equivalent to $\Z$ if and only if
the following two conditions are satisfied:
\begin{enumerate}
\item[(i)] $X$ is coarsely geodesic and has coarse bounded geometry, and
\item[(ii)] the Higson corona $\nu X$ is a topological sum of two indecomposable continua,
 each of which is homeomorphic to $\nu \N$.
\end{enumerate}
\end{theorem}

\begin{proof}
Suppose that $X$ is coarsely equivalent to $\Z$.
Then $X$ satisfies (i)
 since $\Z$ is coarsely equivalent to $\R$.
Since $\Z_{\geq0}$ and $\Z_{<0}$ diverge in $\Z$,
 $\nu \Z $ is a topological sum of $(\Z_{\geq0})^{\ast}$ and $(\Z_{<0})^{\ast}$.
Clearly, both of $\Z_{\geq0}$ and $\Z_{<0}$ are coarsely equivalent to $\N$.
So, both of  $(\Z_{\geq0})^{\ast}$ and $(\Z_{<0})^{\ast}$
 are homeomorphic to $\nu \N$ by Proposition \ref{prop:2.3}.
Thus, $\nu X \approx (\Z_{\geq0})^{\ast}\oplus (\Z_{<0})^{\ast}
 \approx \nu \N \oplus \nu \N$, (ii) is satisfied.
 \par
 
Now suppose that $X$ satisfies (i) and (ii).
Then $X$ cannot satisfy $(\mathfrak{C})$ by Theorem \ref{thm:2.10}.
By Proposition \ref{prop:2.11},
 there are unbounded closed subspaces $A_1,\ A_2\subset X$ such that
\begin{enumerate}
\item\label{C-1} $X=A_1\cup A_2$,
\item\label{C-2} $A_1$ and $A_2$ diverges coarsely in $X$, and
\item\label{C-3} $\nu X \approx \nu A_1 \oplus \nu A_2$.
\end{enumerate}
Then $\nu A_i$ must be an indecomposable continuum by (ii) for each $i=1, 2$.
\par
For $i=1,2$, we shall transform $A_i$ into
 a space $A_{i}'$ which is coarsely geodesic and has coarse bounded geometry.
By Lemma \ref{lem:3.6},
 using our assumption that $X$ is coarsely geodesic,
 there exists a coarse equivalence $f: V(G)\to X$
 from the vertices $V(G)$ of a connected graph $G$,
 where the metric on $V(G)$ is induced by the path metric on $G$
 such that each edge of which is isometric to the unit interval.
We may assume that $f$ is a $\sigma$-rough equivalence for some non-decreasing function $\sigma : \Ray \to \Ray$.
Put 
$$\Gamma =f(V(G))$$
 and say that $\Gamma$ is an $r$-net in $X$ for some $r>0$.
Let $\delta>\sigma (1)$.
Since $A_1$ and $A_2$ diverge coarsely,
 there exists $R>0$ such that
\begin{align}
\tag*{$(3.1)$}\label{3.1}
N(A_1, \delta)\cap N(A_2, \delta)\cap (X\setminus B(x_0, R))=\emptyset.
\end{align}
Note that $G$ is of finite type by Proposition \ref{prop:3.7}.
Thus $f^{-1}(B(x_0, R)\cap \Gamma)$ is a finite set since $f$ is metrically proper.
Since $G$ is a connected graph and $f^{-1}(B(x_0, R)\cap \Gamma)$ is a finite set,
 there exists a connected finite subgraph $F$ of $G$ such that $f^{-1}(B(x_0, R)\cap \Gamma)\subset F$.
Let $G_{0}$ be the {\it induced subgraph of $G$ spanned by the vertex set $V(F)$} of $F$,
 that is,
 the subgraph of $G$ whose vertex set is $V(F)$
 and whose edge set consists of all of the edges in $G$ that have both endpoints in $V(F)$.
We note that $G_0$ is a finite connected subgraph 
 spanned by $V(G_0)=V(F)$.
Put $$\Gamma_i =\Gamma\cap A_i$$ for each $i=1,2$.
Let $H_i$ be the subgraph of $G$ such that
 each edge of which contains a vertex in $f^{-1}(\Gamma_i)\setminus V(G_0)$, $i=1,2$.
Put $$G_i =H_i \cup G_0$$ for each $i=1,2$.
Then $G=G_1 \cup G_2$ since $G_0$ is spanned by $V(G_0)$.
Since each edge of $G$ has diameter $1$ and $\delta >\sigma (1)$,
 $\ref{3.1}$ implies that $H_1\cap H_2 =\emptyset$,
 i.e., $G_1\cap G_2 =G_0$.
Hence,
 both $G_1$ and $G_2$ are connected graphs
 since $G=G_1\cup G_2$ and $G_0=G_1\cap G_2$ are connected.
For each $i=1,2$, put
$$A_{i}'=\overline{N(f(V(G_i) ), r)},$$
 where $V(G_i)$ is the vertices of $G_i$.
Then both $A_{1}'$ and $A_{2}'$ are coarsely geodesic and have coarsely bounded geometry
 since $A_{i}'$ and $G_i $ are coarsely equivalent for each $i=1,2$.
Since $\Gamma$ is an $r$-net in $X$, we have $X=A_{1}'\cup A_{2}'$.
Now it is easy to see that $A_{i}'\setminus B(x_0, R') =A_i \setminus B(x_0, R')$ for sufficiently large $R'>0$.
So we have $A_{i}^{\ast}=(A_{i}')^{\ast}=\nu A_i$, $i=1,2$.
Thus each $\nu A_{i}'$ is an indecomposable continuum for $i=1, 2$.
\par
By Theorem \ref{thm:3.10},
 there exists a coarse equivalence $f_i : \N\to A_{i}'$ for each $i=1,2$.
Define $f: \Z\to X$ by $f(0)=x_0$, $f(k)=f_1(k)$ if $k>0$, and $f(k)=f_2 (|k|)$ if $k<0$.
Then it is easy to check that $f$ is a coarse equivalence between $\Z$ and $X$.
\end{proof}

\begin{corollary}\label{cor:3.13}
Suppose that $X$ is a non-compact proper metric space
 which is coarsely geodesic and has coarse bounded geometry.
Then $X$ is coarsely equivalent to $\Z$ if and only if
 the Higson corona $\nu X$ of $X$ is a topological sum of two indecomposable continua,
 each of which is homeomorphic to $\nu \N$.
\end{corollary}

\section{Higson coronae of finitely generated groups with one or two ends}\label{FGG}

Let $X$ be a metric space.
A {\it proper ray} in $X$ is a proper continuous map $r: \Ray \to X$.
Two proper rays $r_1$ and $r_2$ in $X$ {\it define the same end} of $X$
 if for every compact subset $K \subset X$ there exists $N\in \N$ such that
 $r_1([N, \infty))$ and $r_2([N\infty))$ are contained in the same path component of $X\setminus K$.
This is an equivalence relation on proper rays.
An {\it end} of $X$ is an equivalence class of proper rays in $X$.
\par
For a finitely generated group $G$,
 let $\mbox{Cay}(G,S)$ denote its Cayley graph with respect to a finite generating set $S$.
The number of ends of $G$ is defined as the number of ends of $\mbox{Cay}(G,S)$
 and it does not depend on the choice of finite generating sets.
It is known that a finitely generated group $G$ can have $0$, $1$, $2$, or infinitely many ends,
and the group structure is determined when it has two or infinitely many ends.
Also, the number of ends of $G$ is zero
 if and only if $G$ is a finite group \cite[Theorem 8.32]{BH}.
In that sense, finitely generated groups having exactly one end are fascinating.
Refer to \cite{BH}, \cite{DK}, or \cite{Geoghegan} for more information on ends of groups.
\par
In this section,
 we give characterizations of finitely generated groups that have one or two ends
 by decomposability/indecomposability of the components of their Higson coronae.

The the following characterization was given by Weighill \cite{Wei}.

\begin{theorem}\label{thm:4.1}
\cite[Corollary 7.3]{Wei}
A finitely generated group $G$ with the word length metric has a connected Higson corona
 if and only if $G$ has exactly one end.
\end{theorem}

Now we can refine this result as follows:

\begin{theorem}\label{thm:4.2}
Let $G$ be a finitely generated group with the word length metric.
Then $G$ has exactly one end
 if and only if its Higson corona $\nu G$ is a decomposable continuum.
\end{theorem}

\begin{lemma}\label{lem:4.3}
Let $G$ be a finitely generated infinite group with the word length metric.
Then $G$ is not coarsely equivalent to $\N$.
\end{lemma}

\begin{proof}
Let $G$ be a finitely generated infinite group
 and let $\mbox{Cay}(G)$ be a Cayley graph of $G$ with respect to some finite generating set.
Since $\mbox{Cay}(G)$ contains an isometric copy of $(\Z, d)$,
 it suffices to show
 that $(\Bbb Z, d)$ is not coarsely embeddable
 in $(\N, d)$.
\par
Suppose that $(\Bbb Z, d)$ is {\it coarsely embeddable} in $(\N, d)$,
 that is,
 there exists a coarse map
 $f: \Z \to \N$
 such that
\begin{align}
\tag*{(4.1)}\label{4.3-1}
\rho_{-}(d(x,y))\leq d(f(x), f(y))\leq \rho_{+} (d(x,y))
\end{align}
for every $x, y\in \Z$,
 where $\rho_{-}$ and $\rho_{+}$ are non-decreasing functions 
 with $\displaystyle\lim_{t\to\infty}\rho_{-}(t)=\infty$.
Since $f$ is metrically proper,
 $f(\Z_{\geq 0})$ is an unbounded subset of $\N$.
Also, it follows from \ref{4.3-1} that $f(\Z_{\geq 0})$ is an $\varepsilon$-net in $\N$,
 where $\varepsilon =\max\{ \rho_{+}(1),\ d(1, f(0))\}$.
Hence, for each $k\in \N$,
 there is an element $m_k \in \Z_{\geq 0}$ such that
 $d( f(-k), f(m_k))<\varepsilon$.
Since $\displaystyle\lim_{t\to\infty}\rho_{-}(t)=\infty$,
 there exists $n\in \N$ such that $\rho_{-}(n)>\varepsilon $.
Then we have
\[
\rho_{-}(d(-n, m_{n}))\geq \rho_{-}(n)>\varepsilon  >d( f(-n), f(m_n)),
\]
 a contradiction.
\end{proof}

Note that a finitely generated group with the word length metric is coarsely geodesic
 and has coarse bounded geometry.

\begin{proof}[Proof of Theorem \ref{thm:4.2}]
Suppose that $G$ has exactly one end.
By Theorem \ref{thm:4.1},
 the Higson corona $\nu G$ is a continuum.
Thus $\nu G$ is decomposable by Lemma \ref{lem:4.3} and Theorem \ref{thm:3.10}.
The reverse implication follows from Theorem \ref{thm:4.1}.
\end{proof}

\begin{lemma}\label{lem:4.4}
Let $G$ be a finitely generated group with the word length metric.
Then $G$ has exactly two ends if and only if $G$ is coarsely equivalent to $\Z$.
\end{lemma}

\begin{proof}
Suppose that $G$ has exactly two ends.
Then $G$ contains $\Z$ as a subgroup of finite index \cite[Chap. I, 8.32]{BH}.
Thus $G$ is coarsely equivalent to $\Z$ \cite[Corollary 1.19]{Roe}.
\par
Now let $G$ be a finitely generated group which is coarsely equivalent to $\Z$.
Then the Higson corona $\nu G$ of $G$ is disconnected by Theorem \ref{thm:3.12}.
Thus $G$ has at least two ends by Theorem \ref{thm:4.1}.
Suppose that there are three proper rays
 $r_1, r_2, r_3: \Ray \to \mbox{Cay}(G)$
 that correspond to distinct three ends.
Put $Y=r_1 (\N) \cup r_2 (\N) \cup r_3 (\N)$.
Since $G$ is coarsely equivalent to $\Z$,
 $Y$ is coarsely embeddable in $\Z$,
 that is,
 there exists a coarse map
 $f: Y \to \Z$
 such that
\begin{align}
\tag*{(4.2)}\label{4.4-1}
\rho_{-}(d_G (x,y))\leq d(f(x), f(y))\leq \rho_{+} (d_G (x,y))
\end{align}
for every $x, y\in  Y$,
 where $\rho_{-}$ and $\rho_{+}$ are non-decreasing functions 
 with $\displaystyle\lim_{t\to\infty}\rho_{-}(t)=\infty$.
For each $i=1,2,3$,
 $f(r_i (\N))$ is unbounded in $\Z$ because of properness of $f$.
Hence, we may assume that
 both $\Z_{<0}$ and $\Z_{>0}$ have infinitely many elements of $f(r_1 (\N)\cup r_2 (\N)) $.
Then it follows from \ref{4.4-1} that $f(r_1 (\N)\cup r_2 (\N)) $ is an $\varepsilon$-net in $\Z$,
 where $\varepsilon =\max \{ \rho_{+} (1),\ d(0, f(r_1 (1))),\ d(0, f(r_2 (1)))\}$.
For each $n\in \N$,
 we can take $m_n \in \N$ such that
\[
\min \{ d(f(r_3 (n)), f(r_i (m_n )))\mid i=1, 2\}<\varepsilon.
\]
Since $\displaystyle\lim_{t\to\infty}\rho_{-}(t)=\infty$,
 there is $t_0 \in \R_{+}$ such that $\rho_{-}(t_0 )>\varepsilon$.
For each $i=1,2$,
 we may assume that $d_G (r_3 (n), r_i (m_n ))>t_0$ for sufficiently large $n$
 since $r_3$ and $r_i$ correspond to distinct ends of $G$ (cf. \cite[Lemma 8.28]{BH}).
Then we have
\[
\rho_{-} (d_G (r_3 (n), r_i (m_n ))\geq \rho_{-} (t_0)
>\varepsilon
>d(f(r_3 (n)), f(r_i (m_n ))),
\]
whenever $ d(f(r_3 (n)), f(r_i (m_n )))<\varepsilon$,
a contradiction.
\end{proof}

\begin{theorem}\label{thm:4.5}
Let $G$ be a finitely generated group with the word length metric.
Then $G$ has exactly two ends if and only if its Higson corona $\nu G$ is a topological sum of two non-metrizable indecomposable continua,
 each of which is homeomorphic to $\nu\N$.
\end{theorem}

\begin{proof}
This is a consequence of Theorem \ref{thm:3.12} and Lemma \ref{lem:4.4}.
\end{proof}

\medskip
\begin{center}
\textsc{Acknowledgement}
\end{center}
\par
I wish to express my gratitude to Professor Katsuro Sakai for valuable conversations
 and warm encouragement. 
I am also grateful to the referees for their helpful comments.

\bibliographystyle{amsplain}

\end{document}